\documentclass[12pt]{iopart}

\usepackage{setstack, iopams, mathrsfs, amsopn, amsthm, enumerate}

\theoremstyle{plain}
\newtheorem{thm}{Theorem}[section]
\newtheorem{lem}{Lemma}[section]
\newtheorem{prop}{Proposition}[section]
\newtheorem{cor}{Corollary}[section]

\theoremstyle{definition}
\newtheorem{exmp}{Example}[section]
\newtheorem*{rem}{Remark}

\renewenvironment{proof}[1][Proof.]{\begin{trivlist}
\item[\hskip \labelsep {\bfseries #1}]}{\end{trivlist}}

\eqnobysec

\DeclareMathOperator{\im}{Im}
\DeclareMathOperator*{\Res}{Res}
\newcommand{\const}{\mathrm{const}}

\begin{document}

\title[Inverse problems with parameter in the boundary condition]{Inverse eigenvalue problems for Sturm--Liouville equations with spectral parameter linearly contained in one of the boundary conditions}
\author{Namig J Guliyev}
\address{Institute of Mathematics and Mechanics, National Academy of Sciences of Azerbaijan, 9 F.Agayev str., AZ1141, Baku, Azerbaijan}
\ead{njguliyev@yahoo.com}

\begin{abstract}
  Inverse problems of recovering the coefficients of Sturm--Liouville problems with the eigenvalue parameter linearly contained in one of the boundary conditions are studied:
\begin{enumerate}[1) ]
\item from the sequences of eigenvalues and norming constants;
\item from two spectra.
\end{enumerate}
Necessary and sufficient conditions for the solvability of these inverse problems are obtained.
\end{abstract}

\section{Introduction}

In this paper we consider inverse eigenvalue problems for the equation
\begin{equation}
\label{eq:main}
  \ell y := -y''(x) + q(x)y(x) = \lambda y(x), \quad x \in [0,\pi]
\end{equation}
with the boundary conditions
\begin{eqnarray}
  y'(0) - hy(0) = 0, \label{eq:boundary1} \\
  \lambda (y'(\pi) + H y(\pi)) = H_1 y'(\pi) + H_2 y(\pi), \label{eq:boundary2}
\end{eqnarray}
where $q(x) \in \mathscr{L}_2(0,\pi)$ is a real-valued function, $h, H, H_1, H_2 \in \mathbb{R}$ and
\begin{equation}
\label{eq:def:rho}
  \rho := H H_1 - H_2 > 0.
\end{equation}
Let us denote this problem by $\mathcal{P}(q,h,H,H_1,H_2)$.

Problems with the eigenvalue parameter linearly contained in the boundary conditions have been studied extensively. In~\cite{PRRSEA_77_1977__293-308,MATHZ_133_1973__301-312} an operator-theoretic formulation of the problems of the form~(\ref{eq:main})--(\ref{eq:boundary2}) has been given. It has been shown that one can associate a self-adjoint operator in adequate Hilbert space with such problems whenever the condition~(\ref{eq:def:rho}) holds. Oscillation and comparison results have been obtained in~\cite{PRRSEA_127_1997_6_1123-1136,PROEM2_37_1993_1_57-72,DIFEQ_35_1999_8_1031-1034}. Basis properties and eigenfunction expansions have been considered in~\cite{DIFEQ_33_1997_1_116-120,DIFEQ_37_2001_12_1677-1683,SIBMJ_44_2003_5_813-816,CZEMJ_32(107)_1982__608-622}. Problems with various singularities have been analyzed in~\cite{ACTMH_102_2004_1-2_159-175,PRRSEA_87_1980__1-34}. In the case $\rho < 0$ the problem~(\ref{eq:main})--(\ref{eq:boundary2}) can be associated with a self-adjoint operator in Pontryagin space and not all eigenvalues are necessarily real (see~\cite{MATHN3_60_1996_4_456-458,MATHN3_66_1999_1-2_127-134,PRRSEA_125_1995_6_1205-1218}).

Inverse problems involving linear dependence on the spectral parameter in the boundary conditions have also been investigated. In~\cite{IZVAN__1982_2_15-22} sufficient conditions for two sequences of real numbers to be the spectra of the problems $\mathcal{P}(q,h,0,H_1,H_2)$ and $\mathcal{P}(q,h,0,\widetilde{H_1},\widetilde{H_2})$ are provided, where $H_1 \widetilde{H_2} = \widetilde{H_1} H_2$, $H_2, \widetilde{H_2} > 0$ and $H_1 \ne \widetilde{H_1}$. Various uniqueness theorems are proved in~\cite{JLONM2_62_2000_1_161-182,INVPR_13_1997__1453-1462,NUMFA_24_2003_1-2_85-105,Yurko2002}. We shall provide in this paper another proof of the unique solvability. Numerical techniques are discussed in~\cite{NUMFA_24_2003_1-2_85-105}. In~\cite{MATMA_26_2003_4_349-357,BULLM_33_2001_6_749-757} so called ''almost isospectral'' transformations (i.e., transformations preserving all but finitely many eigenvalues) are studied and using these transformations many direct and inverse results for problems with the spectral parameter in one of the boundary conditions are derived from those for classical Sturm--Liouville problems.

The present paper is devoted to the study of inverse problems by (i) one spectrum and a sequence of norming constants; (ii) two spectra. We obtain necessary conditions for eigenvalues and norming constants in Section~\ref{sec:preliminaries}. In Section~\ref{sec:main_equation} we prove that the kernel of the operator transforming the function $\cos \sqrt{\lambda}x$ to the corresponding solution of the equation~(\ref{eq:main}) satisfies the Gel'fand--Levitan--Marchenko type integral equation.
In Section~\ref{sec:uniqueness} we show that the boundary-value problem~(\ref{eq:main})--(\ref{eq:boundary2}) can be uniquely determined from its spectrum and norming constants.
Reconstruction of the coefficients of the problem from these spectral characteristics is realized in Section~\ref{sec:reconstruction1} using the method analogous to that of Gel'fand and Levitan~\cite{AMEMS2_1_1955__253-304} (see also~\cite{Marchenko77,Yurko2001}).

Sections~\ref{sec:two_spectra} and~\ref{sec:reconstruction2} are devoted to the study of inverse problems by two spectra. In Section~\ref{sec:two_spectra} we consider the problems $\mathcal{P}(q,h,H,H_1,H_2)$ and $\mathcal{P}(q,\widetilde{h},H,H_1,H_2)$ with $h \ne \widetilde{h}$. It's proved that the eigenvalues of two such problems interlace and norming constants of first problem are expressed by these eigenvalues. We use these expressions in Section~\ref{sec:reconstruction2} to solve the inverse problem by two spectra, similarly to the work of Gasymov and Levitan~\cite{USPMN_19_1964_2(116)_3-63} for the classical Sturm--Liouville problems.

\section{Preliminaries}
\label{sec:preliminaries}

Let $\varphi(x, \lambda)$ and $\psi(x, \lambda)$ be the solutions of~(\ref{eq:main}) satisfying the initial conditions
\begin{equation}
\label{eq:phi_psi}
  \fl \varphi(0, \lambda) = 1, \quad \varphi'(0, \lambda) = h, \quad \psi(\pi, \lambda) = -\lambda + H_1, \quad \psi'(\pi, \lambda) = \lambda H - H_2.
\end{equation}
We define
\begin{equation*}
  \chi(\lambda) := \varphi(x, \lambda) \psi'(x, \lambda) - \varphi'(x, \lambda) \psi(x, \lambda),
\end{equation*}
which is independent of $x \in [0,\pi]$. The function $\chi(\lambda)$ is entire and has zeros at the eigenvalues of the problem~(\ref{eq:main})--(\ref{eq:boundary2}). The set of eigenvalues is countable, consists of real numbers and for each eigenvalue $\lambda_n$ there exists such a number $k_n$ that
\begin{equation}
\label{eq:def:k_n}
  \psi(x, \lambda_n) = k_n \varphi(x, \lambda_n), \quad k_n \ne 0.
\end{equation}

In the Hilbert space $\mathcal{H} = \mathscr{L}_2(0,\pi) \oplus \mathbb{C}$ let an inner product be defined by
\begin{equation*}
  (F,G) := \int_0^{\pi} F_1(x)\overline{G_1(x)}\rmd x + \frac{1}{\rho}F_2\overline{G_2}
\end{equation*}
for
\begin{equation*}
  F = \left( {F_1(x) \atop F_2} \right), \quad G = \left( {G_1(x) \atop G_2} \right) \in \mathcal{H}.
\end{equation*}
We define operator (see~\cite{PRRSEA_77_1977__293-308})
\begin{equation*}
  A(F) := \left( { -F''_1(x) +q(x)F_1(x) \atop H_1 F_1'(\pi) + H_2 F_1(\pi) } \right)
\end{equation*}
with
\begin{eqnarray*}
  \fl D(A) = \left\{F \in \mathcal{H} | F_1(x), F'_1(x) \in \mathscr{AC}[0,\pi], \ell F_1 \in \mathscr{L}_2(0,\pi), F'_1(0) - h F_1(0) = 0, \right. \nonumber \\
  \left. F_2 = F_1'(\pi) + H F_1(\pi) \right\}.
\end{eqnarray*}
Then
\begin{equation*}
  \Phi_n := \left( { \varphi(x, \lambda_n) \atop \varphi'(\pi, \lambda_n) + H \varphi(\pi, \lambda_n) } \right)
\end{equation*}
are orthogonal eigenelements of $A$:
\begin{equation*}
  (\Phi_n, \Phi_m) = 0, \quad n \ne m.
\end{equation*}
We also define \emph{norming constants} by
\begin{equation*}
  \gamma_n := \| \Phi_n \| ^2 = \int_0^{\pi} \varphi^2(x, \lambda_n)\rmd x + \frac{\left(\varphi'(\pi, \lambda_n) + H \varphi(\pi, \lambda_n)\right)^2}{\rho}.
\end{equation*}
The numbers $\{\lambda_n, \gamma_n\}_{n \ge 0}$ are called the \emph{spectral data} of the problem~(\ref{eq:main})--(\ref{eq:boundary2}).

\begin{lem}
  The following equality holds:
\begin{equation} \label{eq:dot_chi}
  \dot{\chi}(\lambda_n) = k_n \gamma_n,
\end{equation}
where $\dot{\chi}(\lambda) = \frac{d}{d\lambda} \chi(\lambda)$.
\end{lem}
\begin{proof}
Using~(\ref{eq:phi_psi}) and~(\ref{eq:def:k_n}) in the equality
\begin{equation*}
  \fl (\lambda - \lambda_n) \int_0^{\pi} \psi(x, \lambda) \varphi(x, \lambda_n) \rmd x = \left. (\psi(x, \lambda) \varphi'(x, \lambda_n) - \psi'(x, \lambda) \varphi(x, \lambda_n)) \right|_0^{\pi}
\end{equation*}
we obtain:
\begin{equation*}
  \frac{\chi(\lambda)}{\lambda - \lambda_n} = \int_0^{\pi} \psi(x, \lambda) \varphi(x, \lambda_n) \rmd x + \frac{\rho}{k_n}.
\end{equation*}
As $\lambda \to \lambda_n$ this equality leads to~(\ref{eq:dot_chi}).
\end{proof}
\begin{rem}
  Simplicity of the eigenvalues of~(\ref{eq:main})--(\ref{eq:boundary2}) also follows from this lemma.
\end{rem}
\begin{thm} \label{thm:asymptotics}
  Following asymptotics hold:
\begin{equation}
\label{eq:asymptotics:s_n}
  s_n := \sqrt{\lambda_n} = n - 1 + \frac{\omega}{n \pi} + \frac{\zeta_n}{n}, \quad \{\zeta_n\} \in l_2,
\end{equation}
\begin{equation}
\label{eq:asymptotics:gamma_n}
  \gamma_n = \frac{\pi}{2} + \frac{\zeta'_n}{n}, \quad \{\zeta'_n\} \in l_2,
\end{equation}
where
\begin{equation*}
  \omega = h + H + \frac{1}{2} \int_0^{\pi} q(x)\rmd x.
\end{equation*}
\end{thm}
\begin{proof}
We denote $s := \sqrt{\lambda}$. Then from the asymptotic estimates (see~\cite{Marchenko77,Yurko2001})
\begin{eqnarray*}
  \fl \varphi(x, \lambda) = \cos sx + \left( h + \frac{1}{2} \int_0^x q(t)\rmd t \right) \frac{\sin sx}{s} + \frac{1}{2} \int_0^x q(t)\frac{\sin s(x-2t)}{s}\rmd t \\
  + O\left(\frac{\rme^{|\im sx|}}{|s|^2}\right),
\end{eqnarray*}
\begin{eqnarray*}
  \fl \varphi'(x, \lambda) = -s \sin sx + \left( h + \frac{1}{2} \int_0^x q(t)\rmd t \right) \cos sx + \frac{1}{2} \int_0^x q(t)\cos s(x-2t)\rmd t \\
  + O\left(\frac{\rme^{|\im sx|}}{|s|}\right)
\end{eqnarray*}
using~(\ref{eq:phi_psi}) we have:
\begin{equation}
\label{eq:asymptotics:chi}
  \chi(\lambda) = -s^3 \sin s\pi + \left( h + H + \frac{1}{2} \int_0^{\pi} q(x)\rmd x \right) s^2 \cos s\pi + I(s) s^2,
\end{equation}
where
\begin{equation*}
  I(s) = \frac{1}{2} \int_0^{\pi} q(t)\cos s(\pi-2t)\rmd t + O\left(\frac{\rme^{|\im s\pi|}}{|s|}\right).
\end{equation*}
Now using Bessel's inequality it's easy to obtain~(\ref{eq:asymptotics:s_n}) and~(\ref{eq:asymptotics:gamma_n}).
\end{proof}

Since the function $\chi(\lambda)$ is entire of order $1/2$, from Hadamard's theorem(see \cite[Section~4.2]{Levin96}), using~(\ref{eq:asymptotics:chi}) we obtain:
\begin{equation*}
  \chi(\lambda) = -\pi(\lambda-\lambda_0)(\lambda-\lambda_1) \prod_{n=2}^{\infty} \frac{\lambda_n - \lambda}{(n-1)^2}.
\end{equation*}

\section{Main Equation}
\label{sec:main_equation}

\begin{thm} \label{thm:expansion:direct}
  Let $f(x) \in \mathscr{AC}[0,\pi]$. Then
\begin{equation*}
  f(x) = \sum_{n=0}^{\infty} \left( \frac{1}{\gamma_n} \int_0^{\pi} f(t) \varphi(t, \lambda_n) \rmd t \right) \varphi(x, \lambda_n)
\end{equation*}
with uniform convergence in $[0,\pi]$.
\end{thm}
\begin{proof}
  We denote
\begin{equation*}
  G(x, t, \lambda) := \frac{1}{\chi(\lambda)}\left\{
\begin{array}{ll}
  \varphi (x, \lambda) \psi (t, \lambda), & 0 \le x \le t \le \pi \\
  \psi (x, \lambda) \varphi (t, \lambda), & 0 \le t \le x \le \pi
\end{array}
\right.
\end{equation*}
and consider the function
\begin{eqnarray*}
  \fl Y(x, \lambda) := \int_0^{\pi} G(x, t, \lambda) f(t) \rmd t \\
  = \frac{1}{\chi(\lambda)} \left( \psi (x, \lambda) \int_0^x \varphi (t, \lambda) f(t) \rmd t + \varphi (x, \lambda) \int_x^{\pi} \psi (t, \lambda) f(t) \rmd t \right).
\end{eqnarray*}
Using~(\ref{eq:def:k_n}) and~(\ref{eq:dot_chi}) we obtain:
\begin{eqnarray*}
  \fl \Res_{\lambda = \lambda_n} Y(x, \lambda)
  = \frac{1}{\dot{\chi}(\lambda_n)} \left( \psi (x, \lambda_n) \int_0^x \varphi (t, \lambda_n) f(t) \rmd t + \varphi (x, \lambda_n) \int_x^{\pi} \psi (t, \lambda_n) f(t) \rmd t \right) \\
  \lo{=} \frac{k_n}{\dot{\chi}(\lambda_n)} \varphi (x, \lambda_n) \int_0^{\pi} \varphi (t, \lambda_n) f(t) \rmd t = \frac{1}{\gamma_n} \varphi (x, \lambda_n) \int_0^{\pi} \varphi (t, \lambda_n) f(t) \rmd t.
\end{eqnarray*}
Noting that $\varphi(x, \lambda)$ and $\psi(x, \lambda)$ are solutions of~(\ref{eq:main}) and integrating by parts we can write:
\begin{equation*}
  Y(x, \lambda) = \frac{f(x)}{\lambda} + \frac{Z(x, \lambda)}{\lambda},
\end{equation*}
where
\begin{eqnarray*}
  \fl Z(x, \lambda) = \frac{1}{\chi(\lambda)} \left( \psi(x, \lambda) \int_0^x \varphi'(t, \lambda) f'(t) \rmd t + \varphi(x, \lambda) \int_x^{\pi} \psi'(t, \lambda) f'(t) \rmd t \right. \\
  + h f(0) \psi (x, \lambda) - (\lambda H - H_2) f(\pi) \varphi (x, \lambda) \\
  + \left. \psi (x, \lambda) \int_0^x \varphi (t, \lambda) q(t) f(t) \rmd t + \varphi (x, \lambda) \int_x^{\pi} \psi (t, \lambda) q(t) f(t) \rmd t \right).
\end{eqnarray*}
Using asymptotic estimates for the functions $\varphi(x, \lambda)$, $\psi(x, \lambda)$ and $\chi(\lambda)$ the following equality can be proved:
\begin{equation*}
  \lim _{{|s| \to \infty \atop s \in G_{\delta}}} \max _{0 \le x \le \pi} |Z(x, \lambda)| = 0,
\end{equation*}
where $G_{\delta} = \{ s: |s-n| \ge \delta, n = 0, \pm 1, \pm 2, \dots \}$ for some small fixed $\delta > 0$.

Now consider the contour integral
\begin{equation*}
  I_N(x) = \frac{1}{2 \pi i} \int _{C_N} Y(x, \lambda) \rmd \lambda,
\end{equation*}
where $C_N = \{ \lambda: |\lambda| = (N - 1/2)^2 \}$. From the above equalities we have:
\begin{equation*}
  I_N(x) = f(x) + \varepsilon _N(x), \quad \lim _{N \to \infty} \max _{0 \le x \le \pi} |\varepsilon _N(x)| = 0.
\end{equation*}
On the other hand, using the residue calculus we obtain:
\begin{equation*}
  I_N(x) = \sum_{n=0}^{N} \left( \frac{1}{\gamma_n} \int_0^{\pi} f(t) \varphi(t, \lambda_n) \rmd t \right) \varphi(x, \lambda_n).
\end{equation*}
From last two equalities we obtain the statement of the theorem.
\end{proof}
\begin{lem} \label{lem:a_x} (cf.~\cite[Lemma~1.5.4]{Yurko2001})
  Assume that numbers $\{s_n, \gamma_n\}_{n \ge 0}$ satisfying the conditions~(\ref{eq:asymptotics:s_n}), (\ref{eq:asymptotics:gamma_n}) and $\gamma_n \ne 0$ are given and denote
\begin{equation*}
  a(x) := \sum_{n=1}^{\infty} \left( \frac{\cos s_n x}{\gamma_n} - \frac{\cos (n-1)x}{\alpha_{n-1}^0} \right),
\end{equation*}
where
\begin{equation*}
  \alpha_n^0 = \left\{
\begin{array}{ll}
  \frac{\displaystyle \pi}{\displaystyle 2}, & n \ge 1, \\
  \pi, & n = 0.
\end{array}
\right.
\end{equation*}
Then $a(x) \in \mathscr{W}_2^1(0, 2\pi)$.
\end{lem}

We denote
\begin{equation}
\label{eq:def:F}
  \fl F(x,t) = \frac{\cos s_0 x \cos s_0 t}{\gamma_0} + \sum_{n=1}^{\infty} \left( \frac{\cos s_n x \cos s_n t}{\gamma_n} - \frac{\cos (n-1)x \cos (n-1)t}{\alpha_{n-1}^0} \right).
\end{equation}
Since
\begin{equation*}
  F(x,t) = \frac{\cos s_0 x \cos s_0 t}{\gamma_0} + \frac{a(x+t) + a(x-t)}{2},
\end{equation*}
Lemma~\ref{lem:a_x} implies that $F(x,t)$ is continuous and $\frac{d}{dx}F(x,x) \in \mathscr{L}_2(0,\pi)$.
Using the transformation operators (\cite{Marchenko77,Yurko2001}), we can write equalities
\begin{equation}
\label{eq:transformation:phi}
  \varphi(x, \lambda) = \cos sx + \int_0^x K(x,t)\cos st \rmd t,
\end{equation}
\begin{equation}
\label{eq:transformation:cos}
  \cos sx = \varphi(x, \lambda) + \int_0^x H(x,t)\varphi(t, \lambda) \rmd t,
\end{equation}
where $K(x,t)$ and $H(x,t)$ are real-valued continuous functions and
\begin{equation}
\label{eq:K_x_x}
  K(x,x) = h + \frac{1}{2} \int_0^x q(t) \rmd t.
\end{equation}

\begin{thm}
  For each fixed $x \in (0,\pi]$ the kernel $K(x,t)$ satisfies the following equation:
\begin{equation}
\label{eq:main_equation}
  F(x,t) + K(x,t) + \int_0^x K(x, \tau) F(\tau, t) \rmd \tau = 0, \quad 0 < t < x.
\end{equation}
\end{thm}
\begin{proof}
  Using equalities~(\ref{eq:transformation:phi}) and~(\ref{eq:transformation:cos}) we obtain:
\begin{equation*}
  \fl \sum_{n=0}^N \frac{\varphi(x, \lambda_n) \cos s_n t}{\gamma_n} = \sum_{n=0}^N \left( \frac{\cos s_n x \cos s_n t}{\gamma_n} + \frac{\cos s_n t}{\gamma_n} \int_0^x K(x, \tau) \cos s_n \tau \rmd \tau \right),
\end{equation*}
\begin{equation*}
  \fl \sum_{n=0}^N \frac{\varphi(x, \lambda_n) \cos s_n t}{\gamma_n} = \sum_{n=0}^N \left( \frac{\varphi(x, \lambda_n) \varphi(t, \lambda_n)}{\gamma_n} + \frac{\varphi(x, \lambda_n)}{\gamma_n} \int_0^t H(t, \tau) \varphi(\tau, \lambda_n) \rmd \tau \right).
\end{equation*}
Therefore we can write:
\begin{equation*}
  \Phi_N(x,t) = I_N(x,t) + I_N'(x,t) + I_N''(x,t) + I_N'''(x,t),
\end{equation*}
where
\begin{equation*}
  \Phi_N(x,t) = \sum_{n=0}^N \frac{\varphi(x, \lambda_n) \varphi(t, \lambda_n)}{\gamma_n} - \sum_{n=0}^{N-1} \frac{\cos nx \cos nt}{\alpha_n^0},
\end{equation*}
\begin{equation*}
  I_N(x,t) = \sum_{n=0}^N \frac{\cos s_n x \cos s_n t}{\gamma_n} - \sum_{n=0}^{N-1} \frac{\cos nx \cos nt}{\alpha_n^0},
\end{equation*}
\begin{equation*}
  I_N'(x,t) = \sum_{n=0}^{N-1} \frac{\cos nt}{\alpha_n^0} \int_0^x K(x, \tau) \cos n\tau \rmd \tau,
\end{equation*}
\begin{equation*}
  I_N''(x,t) = \int_0^x K(x, \tau) \left( \sum_{n=0}^N \frac{\cos s_n t \cos s_n \tau}{\gamma_n} - \sum_{n=0}^{N-1} \frac{\cos nt \cos n\tau}{\alpha_n^0} \right) \rmd \tau,
\end{equation*}
\begin{equation*}
  I_N'''(x,t) = -\sum_{n=0}^N \frac{\varphi(x, \lambda_n)}{\gamma_n} \int_0^t H(t, \tau) \varphi(\tau, \lambda_n) \rmd \tau.
\end{equation*}
Let $f(x)$ be an absolutely continuous function. Then using Theorem~\ref{thm:expansion:direct} we obtain (uniformly on $x \in [0, \pi]$):
\begin{equation*}
  \lim_{N \to \infty} \int_0^{\pi} f(t) \Phi_N(x,t) \rmd t = 0,
\end{equation*}
\begin{equation*}
  \lim_{N \to \infty} \int_0^{\pi} f(t) I_N(x,t) \rmd t = \int_0^{\pi} f(t)F(x,t)\rmd t,
\end{equation*}
\begin{equation*}
  \lim_{N \to \infty} \int_0^{\pi} f(t) I_N'(x,t) \rmd t = \int_0^x f(t)K(x,t)\rmd t,
\end{equation*}
\begin{equation*}
  \lim_{N \to \infty} \int_0^{\pi} f(t) I_N''(x,t) \rmd t = \int_0^{\pi} f(t) \left( \int_0^x K(x, \tau) F(\tau, t) \rmd \tau \right) \rmd t,
\end{equation*}
\begin{equation*}
  \lim_{N \to \infty} \int_0^{\pi} f(t) I_N'''(x,t) \rmd t = -\int_x^{\pi} f(t)H(t,x)\rmd t,
\end{equation*}
We put $K(x,t) = H(x,t) = 0$ for $x < t$. Since $f(x)$ can be chosen arbitrarily, we have
\begin{equation*}
  F(x,t) + K(x,t) + \int_0^x K(x, \tau) F(\tau, t) \rmd \tau - H(t,x) = 0.
\end{equation*}
When $t < x$ this equation implies~(\ref{eq:main_equation}).
\end{proof}

\section{Uniqueness}
\label{sec:uniqueness}

\begin{lem} \label{lem:main_equation_uniqueness}
  For each fixed $x \in (0,\pi]$ equation~(\ref{eq:main_equation}) has a unique solution $K(x,t) \in \mathscr{L}_2(0,x)$.
\end{lem}
\begin{proof}
  It suffices to prove that homogeneous equation
\begin{equation*}
  g(t) + \int_0^x F(\tau,t)g(\tau) \rmd \tau = 0
\end{equation*}
has only trivial solution $g(t) = 0$.

Let $g(t)$ be a solution of the above equation and $g(t) = 0$ for $t \in (x,\pi)$. Then
\begin{equation*}
  \int_0^x g^2(t) \rmd t + \int_0^x \int_0^x F(\tau,t)g(\tau)g(t) \rmd \tau \rmd t = 0
\end{equation*}
or
\begin{equation*}
  \fl \int_0^x g^2(t) \rmd t + \sum_{n=0}^{\infty} \frac{1}{\gamma_n} \left( \int_0^x g(t) \cos s_nt \rmd t \right)^2 - \sum_{n=0}^{\infty} \frac{1}{\alpha_n^0} \left( \int_0^x g(t) \cos nt \rmd t \right)^2 = 0.
\end{equation*}
Using Parseval's equality
\begin{equation*}
  \int_0^x g^2(t) \rmd t = \sum_{n=0}^{\infty} \frac{1}{\alpha_n^0} \left( \int_0^x g(t) \cos nt \rmd t \right)^2
\end{equation*}
and noting that $\gamma_n > 0$ we obtain:
\begin{equation*}
  \int_0^x g(t) \cos s_nt \rmd t = 0, \quad n \ge 0.
\end{equation*}
The system $\{ \cos s_nt \}$ is complete in $\mathscr{L}_2(0,\pi)$ (see~\cite[Proposition~1.8.6]{Yurko2001}). Therefore $g(t) = 0$.
\end{proof}
\begin{lem} \label{lem:smoothness} (\cite[Lemma~1.5.2]{Yurko2001})
  Consider an integral equation
\begin{equation}
\label{eq:lem:smoothness}
  y(t,\alpha) + \int_a^b A(t,\tau,\alpha)y(\tau,\alpha)\rmd \tau = f(t,\alpha), \quad a \le t \le b,
\end{equation}
where $A(t,\tau,\alpha)$ and $f(t,\alpha)$ are continuous functions. Assume that, for some fixed $\alpha = \alpha_0$ the homogeneous equation
\begin{equation*}
  z(t) + \int_a^b A_0(t,\tau)z(\tau)\rmd \tau = 0, \quad A_0(t,\tau) := A(t,\tau,\alpha_0)
\end{equation*}
has only trivial solution. Then in some neighbourhood of the point $\alpha = \alpha_0$ the solution $y(t,\alpha)$ of the equation~(\ref{eq:lem:smoothness}) is continuous on $t$ and $\alpha$. Moreover, the function $y(t,\alpha)$ has the same smoothness as $A(t,\tau,\alpha)$ and $f(t,\alpha)$.
\end{lem}
\begin{thm}
  Let $\mathcal{P}(q,h,H,H_1,H_2)$ and $\mathcal{P}(\widetilde{q}, \widetilde{h}, \widetilde{H}, \widetilde{H}_1, \widetilde{H}_2)$ be two boundary-value problems with one boundary condition depending linearly on the spectral parameter and
\begin{equation*}
  \lambda_n = \widetilde{\lambda}_n, \quad \gamma_n = \widetilde{\gamma}_n, \qquad \qquad n \ge 0.
\end{equation*}
Then
\begin{equation*}
  \fl q(x) = \widetilde{q}(x)\textit{ a.e. on }(0,\pi), \quad h = \widetilde{h}, \quad H = \widetilde{H}, \quad H_1 = \widetilde{H}_1, \quad H_2 = \widetilde{H}_2.
\end{equation*}
\end{thm}
\begin{proof}
  According to the formula~(\ref{eq:def:F}) $F(x,t) = \widetilde{F}(x,t)$. Then from the main equation~(\ref{eq:main_equation}) we obtain $K(x,t) = \widetilde{K}(x,t)$. Equality~(\ref{eq:K_x_x}) implies that $h = \widetilde{h}$ and $q(x) = \widetilde{q}(x)$ a.e. on $(0,\pi)$. From~(\ref{eq:transformation:phi}) we have $\varphi(x, \lambda_n) = \widetilde{\varphi}(x, \lambda_n)$. In consideration of~(\ref{eq:asymptotics:chi}) we obtain $\chi(\lambda) \equiv \widetilde{\chi}(\lambda)$ and $k_n = \widetilde{k}_n$. Finally, by using~(\ref{eq:phi_psi}) and~(\ref{eq:def:k_n}) the remaining part of the theorem can be proved.
\end{proof}

\section{Reconstruction by spectral data}
\label{sec:reconstruction1}

Let two sequences of real numbers $\{\lambda_n\}$ and $\{\gamma_n\}$ ($n \in \mathbb{Z}_{+}$) with the following properties be given:
\begin{equation}
\label{eq:spectral_data1}
  s_n = \sqrt{\lambda_n} = n - 1 + \frac{\omega}{n \pi} + \frac{\zeta_n}{n}, \quad
  \gamma_n = \frac{\pi}{2} + \frac{\zeta'_n}{n}, \quad
  \{\zeta_n\}, \{\zeta'_n\} \in l_2,
\end{equation}
\begin{equation}
\label{eq:spectral_data2}
  \lambda_n \ne \lambda_m, \quad n \ne m, \qquad \gamma_n > 0, \quad n \in \mathbb{Z}_{+}.
\end{equation}

Using these numbers we construct $F(x,t)$ by the formula~(\ref{eq:def:F}) and determine $K(x,t)$ from~(\ref{eq:main_equation}). Substituting $t \to tx$, $\tau \to \tau x$ in~(\ref{eq:main_equation}) we obtain:
\begin{equation*}
  F(x,xt) + K(x,xt) + x \int_0^1 K(x, x\tau) F(x\tau, xt) \rmd \tau = 0, \quad 0 \le t \le 1.
\end{equation*}
According to this equation and Lemmas~\ref{lem:main_equation_uniqueness} and~\ref{lem:smoothness} $K(x,t)$ is determined uniquely and $\frac{d}{dx}K(x,x) \in \mathscr{L}_2(0,\pi)$. Now, let's construct the functions $q(x)$, $\varphi(x, \lambda)$, $\chi(\lambda)$ and the number $h$ by
\begin{equation*}
  q(x) := 2\frac{d}{dx}K(x,x), \quad h := K(0,0),
\end{equation*}
\begin{equation} \label{eq:inverse:phi}
  \varphi(x, \lambda) := \cos sx + \int_0^x K(x,t)\cos st \rmd t,
\end{equation}
\begin{equation*}
  \chi(\lambda) = -\pi(\lambda-\lambda_0)(\lambda-\lambda_1) \prod_{n=2}^{\infty} \frac{\lambda_n - \lambda}{(n-1)^2}
\end{equation*}
and put
\begin{equation*}
  k_n := \frac{\dot{\chi}(\lambda_n)}{\gamma_n}.
\end{equation*}
From (\ref{eq:spectral_data2}) we have: $k_n \ne 0$.
\begin{lem} \label{lem:Marchenko} (\cite[Lemma~3.4.2]{Marchenko77})
  For the functions $u(z)$, $v(z)$ to be represented in the form
\begin{equation*}
  u(z) = \sin \pi z + A\pi\frac{4z}{4z^2-1}\cos \pi z + \frac{f(z)}{z},
\end{equation*}
\begin{equation*}
  v(z) = \cos \pi z - B\pi\frac{\sin \pi z}{z} + \frac{g(z)}{z},
\end{equation*}
where
\begin{equation*}
  f(z) = \int_0^{\pi} \widetilde{f}(t) \cos zt \rmd t, \quad \widetilde{f}(t) \in \mathscr{L}_2[0,\pi], \quad \int_0^{\pi} \widetilde{f}(t)\rmd t = 0,
\end{equation*}
\begin{equation*}
  g(z) = \int_0^{\pi} \widetilde{g}(t) \sin zt \rmd t, \quad \widetilde{g}(t) \in \mathscr{L}_2[0,\pi],
\end{equation*}
it is necessary and sufficient to have the form
\begin{equation*}
  u(z) = \pi z \prod_{n=1}^{\infty} n^{-2}(u_n^2 - z^2), \quad u_n = n - \frac{A}{n} + \frac{\alpha_n}{n},
\end{equation*}
\begin{equation*}
  v(z) = \prod_{n=1}^{\infty} \left( n - \frac{1}{2} \right)^{-2}(v_n^2 - z^2), \quad v_n = n - \frac{1}{2} - \frac{B}{n} + \frac{\beta_n}{n},
\end{equation*}
where $\alpha_n$ and $\beta_n$ are arbitrary sequences that satisfy conditions
\begin{equation*}
  \sum_{n=1}^{\infty} |\alpha_n|^2 < \infty, \quad \sum_{n=1}^{\infty} |\beta_n|^2 < \infty.
\end{equation*}
\end{lem}
\begin{lem} \label{lem:zero_sum}
The following equality holds:
\begin{equation}
\label{eq:zero_sum}
  \sum_{n=0}^{\infty} \frac{\varphi(x, \lambda_n)}{k_n \gamma_n} = 0
\end{equation}
\end{lem}
\begin{proof}
Using the residue calculus we get:
\begin{equation*}
  \fl \sum_{n=0}^{N} \frac{\varphi(x, \lambda_n)}{k_n \gamma_n} = \sum_{n=0}^{N} \frac{\varphi(x, \lambda_n)}{\dot{\chi}(\lambda_n)} = \sum_{n=0}^{N} \Res_{\lambda=\lambda_n} \frac{\varphi(x, \lambda)}{\chi(\lambda)} = \frac{1}{2 \pi i} \int_{C_N} \frac{\varphi(x, \lambda)}{\chi(\lambda)} \rmd \lambda,
\end{equation*}
where $C_N = \{ \lambda: |\lambda| = (N-1/2)^2 \}$. Lemma~\ref{lem:Marchenko} implies that
\begin{equation*}
  \chi(\lambda) = -s^3 \sin s\pi + O\left( |s|^2 \rme^{|\im s\pi|} \right).
\end{equation*}
We denote $G_{\delta} = \{ s: |s-n| \ge \delta, n = 0, \pm 1, \pm 2, \dots \}$ for some small fixed $\delta > 0$ and recall that (see e.g. \cite[p.~15]{Yurko2001})
\begin{equation*}
  | \sin s\pi | \ge C_{\delta} \rme^{|\im s\pi|}, \quad s \in G_{\delta},
\end{equation*}
where $C_{\delta}$ does not depend on $s$. Therefore we obtain
\begin{equation*}
  | \chi(\lambda) | \ge C_{\delta} |s|^3 \rme^{|\im s\pi|}, \quad s \in G_{\delta}, |s| \ge s_{\delta}
\end{equation*}
for sufficiently large $s_{\delta}$. Since
\begin{equation*}
  | \varphi(x, \lambda) | = O\left( \rme^{|\im s\pi|} \right), \quad x \in [0,\pi],
\end{equation*}
we obtain~(\ref{eq:zero_sum}).
\end{proof}
As in the theory of classical Sturm--Liouville problems (see \cite[Lemma~1.5.8 and Corollary~1.5.1]{Yurko2001}) the following lemmas can be proved:
\begin{lem} \label{lem:inverse:solution}
The following relations hold:
\begin{eqnarray*}
  -\varphi''(x, \lambda) + q(x)\varphi(x, \lambda) = \lambda \varphi(x, \lambda), \\
  \varphi(0, \lambda) = 1, \quad \varphi'(0, \lambda) = h.
\end{eqnarray*}
\end{lem}
\begin{lem} \label{lem:inverse:Parseval}
  For any $f(x), g(x) \in \mathscr{L}_2(0,\pi)$ the following equality holds:
\begin{equation*}
  \fl \int_0^{\pi} f(x)g(x)\rmd x = \sum_{n=0}^{\infty} \frac{1}{\gamma_n} \left( \int_0^{\pi} f(t) \varphi(t, \lambda_n) \rmd t \right) \left( \int_0^{\pi} g(t) \varphi(t, \lambda_n) \rmd t \right).
\end{equation*}
\end{lem}
\begin{prop} \label{prop:basis:cos} (\cite[Proposition~1.8.6]{Yurko2001})
  Let numbers $\{ \rho_n \}_{n \ge 0}$, $\rho_n^2 \ne \rho_k^2$ $(n \ne k)$ of the form
\begin{equation*}
  \rho_n = n + \frac{a}{n} + \frac{\xi_n}{n}, \quad \{\xi_n\} \in l_2, \quad a \in \mathbb{C}
\end{equation*}
be given. Then the sequence $\{ \cos \rho_n x \}_{n \ge 0}$ forms a Riesz basis in the space $\mathscr{L}_2(0,\pi)$.
\end{prop}
\begin{cor} \label{cor:basis:inverse}
  For any fixed $n_0 \in \mathbb{Z}_{+}$ the system $\{ \varphi(x, \lambda_n) \} (n \ne n_0)$ forms a Riesz basis in the space $\mathscr{L}_2(0,\pi)$.
\end{cor}
\begin{proof}
  According to Lemma~\ref{lem:inverse:solution} we can write representation of the form~(\ref{eq:transformation:cos}). Therefore, there is one-to-one correspondence between expansions in $\{ \cos s_n x \} (n \ne n_0)$ and $\{ \varphi(x, \lambda_n) \} (n \ne n_0)$.
\end{proof}
\begin{lem}
  For any $f(x) \in \mathscr{W}_2^2(0,\pi)$, the expansion
\begin{equation}
\label{eq:expansion:inverse}
  f(x) = \sum_{n=0}^{\infty} \left( \frac{1}{\gamma_n} \int_0^{\pi} f(t) \varphi(t, \lambda_n) \rmd t \right) \varphi(x, \lambda_n)
\end{equation}
holds.
\end{lem}
\begin{proof}
  Consider the series
\begin{equation}
\label{eq:inverse:f*}
  f^{*}(x) = \sum_{n=0}^{\infty} c_n \varphi(x, \lambda_n),
\end{equation}
where
\begin{equation*}
  c_n := \frac{1}{\gamma_n} \int_0^{\pi} f(t) \varphi(t, \lambda_n) \rmd t.
\end{equation*}
Using Lemma~\ref{lem:inverse:solution} and integrating by parts we obtain:
\begin{eqnarray*}
  \fl c_n = \frac{1}{\gamma_n \lambda_n} \int_0^{\pi} f(t) \left( -\varphi''(t, \lambda_n) + q(t) \varphi(t, \lambda_n) \right) \rmd t \\
  \lo{=} \frac{1}{\gamma_n \lambda_n} \left( h f(0) - f'(0) + f'(\pi) \varphi(\pi, \lambda_n) - f(\pi) \varphi'(\pi, \lambda_n) \right) \\
  + \frac{1}{\gamma_n \lambda_n} \int_0^{\pi} \varphi(t, \lambda_n) ( -f''(t) + q(t)f(t)) \rmd t.
\end{eqnarray*}
We can easily prove that as $n \to \infty$
\begin{equation*}
  c_n = O\left( \frac{1}{n^2} \right), \quad \varphi(t, \lambda_n) = O(1)
\end{equation*}
uniformly on $t \in [0,\pi]$. Therefore, the series~(\ref{eq:inverse:f*}) converges absolutely and uniformly on $x \in [0,\pi]$. According to Lemma~\ref{lem:inverse:Parseval}
\begin{eqnarray*}
  \fl \int_0^{\pi} f(x)g(x)\rmd x = \sum_{n=0}^{\infty} c_n \int_0^{\pi} g(t) \varphi(t, \lambda_n) \rmd t = \int_0^{\pi} g(t) \sum_{n=0}^{\infty} c_n \varphi(t, \lambda_n) \rmd t \\
  = \int_0^{\pi} g(t)f^{*}(t)\rmd t.
\end{eqnarray*}
Since $g(x)$ can be chosen arbitrarily, we conclude that $f^{*}(x) = f(x)$.
\end{proof}

We can write~(\ref{eq:zero_sum}) as
\begin{equation*}
  \frac{\varphi(x, \lambda_{n_0})}{\gamma_{n_0}} = -\sum_{n \ne n_0} \frac{k_{n_0} \varphi(x, \lambda_n)}{k_n \gamma_n}
\end{equation*}
for any $n_0 \in \mathbb{Z}_{+}$. Let $m \ne n_0$ be any fixed number and $f(x) = \varphi(x, \lambda_m)$. Then using the above equality in~(\ref{eq:expansion:inverse}), we have:
\begin{equation*}
  \varphi(x, \lambda_m) = \sum_{n \ne n_0} c_{mn} \varphi(x, \lambda_n),
\end{equation*}
where
\begin{equation*}
  c_{mn} = \frac{1}{\gamma_n} \int_0^{\pi} \varphi(t, \lambda_m) \left( \varphi(t, \lambda_n) - \frac{k_{n_0}}{k_n} \varphi(t, \lambda_{n_0}) \right) \rmd t.
\end{equation*}
Corollary~\ref{cor:basis:inverse} implies $c_{mn} = \delta_{mn}$. Here $\delta_{mn}$ is the Kronecker delta. In other words, denoting $a_{mn} := \int_0^{\pi} \varphi(t, \lambda_m) \varphi(t, \lambda_n) \rmd t$, we have:
\begin{equation*}
  a_{mm} - \frac{k_n}{k_m}a_{mn} = \gamma_m, \quad m \ne n.
\end{equation*}
It's clear (from definition) that $a_{mn} = a_{nm}$. Using these relations we calculate:
\begin{equation*}
  \fl k_m^2 (\gamma_m - a_{mm}) = - k_m k_n a_{mn} = - k_n k_m a_{nm} = k_n^2 (\gamma_n - a_{nn}), \quad m \ne n.
\end{equation*}
Therefore $k_n^2 (\gamma_n - a_{nn}) = \const$. Let's denote this constant by $\rho$. Then we have:
\begin{equation*}
  \fl \int_0^{\pi} \varphi^2(t, \lambda_n) \rmd t = \gamma_n - \frac{\rho}{k_n^2}, \qquad \int_0^{\pi} \varphi(t, \lambda_m) \varphi(t, \lambda_n) \rmd t = - \frac{\rho}{k_m k_n}, \quad m \ne n.
\end{equation*}
Now using the equality
\begin{equation*}
  \varphi(\pi, \lambda) \varphi'(\pi, \mu) - \varphi'(\pi, \lambda) \varphi(\pi, \mu) = (\lambda - \mu) \int_0^{\pi} \varphi(t, \lambda) \varphi(t, \mu) \rmd t
\end{equation*}
we write:
\begin{equation*}
  \frac{k_n \varphi(\pi, \lambda_n) k_m \varphi'(\pi, \lambda_m) - k_n \varphi'(\pi, \lambda_n) k_m \varphi(\pi, \lambda_m)}{\lambda_n - \lambda_m} = -\rho, \quad n \ne m.
\end{equation*}
Denoting $A_n := k_n \varphi(\pi, \lambda_n)$, $B_n := k_n \varphi'(\pi, \lambda_n)$ we can write the above equality as
\begin{equation}
\label{eq:A_B}
  A_n B_m - B_n A_m = \rho ( \lambda_m - \lambda_n ), \quad n \ne m.
\end{equation}
Let $i$, $j$, $m$ and $n$ be pairwise distinct nonnegative integers. By summing the equalities
\begin{equation*}
  A_n B_m - B_n A_m = \rho ( \lambda_m - \lambda_n ),
\end{equation*}
\begin{equation*}
  A_m B_i - B_m A_i = \rho ( \lambda_i - \lambda_m ),
\end{equation*}
\begin{equation*}
  A_i B_n - B_i A_n = \rho ( \lambda_n - \lambda_i ),
\end{equation*}
we have:
\begin{equation*}
  A_n (B_m - B_i) + B_n (A_i - A_m) = B_m A_i - A_m B_i.
\end{equation*}
Writing this equality again, but this time with $n$ replaced by $j$ and subtracting them we finally obtain:
\begin{equation*}
  (A_n - A_j)(B_m - B_i) = (A_m - A_i)(B_n - B_j).
\end{equation*}

If $B_n = B_j$ for some $n,j \in \mathbb{Z}_{+}$, then $B_n = \const$. In this case~(\ref{eq:A_B}) implies $A_n = \varkappa_1 \lambda_n + \varkappa_2$ with some constants $\varkappa_1$ and $\varkappa_2$. Continuing this procedure for the case $B_n \ne B_j$ we obtain $A_n = \varkappa_1 \lambda_n + \varkappa_2$ and $B_n = \varkappa_3 \lambda_n + \varkappa_4$.

So in both cases
\begin{equation*}
  k_n \varphi(\pi, \lambda_n) = \varkappa_1 \lambda_n + \varkappa_2, \quad k_n \varphi'(\pi, \lambda_n) = \varkappa_3 \lambda_n + \varkappa_4.
\end{equation*}
Using~(\ref{eq:spectral_data1}),~(\ref{eq:inverse:phi}) and Lemma~\ref{lem:Marchenko} we calculate:
\begin{equation*}
  \fl k_n = (-1)^n n^2 + O(n), \quad \varphi(\pi, \lambda_n) = (-1)^{n-1} + O(\frac{1}{n}), \quad \lambda_n = n^2 + O(n).
\end{equation*}
Therefore $\varkappa_1 = -1$. Denoting $H_1 := \varkappa_2$, $H := -\varkappa_3$, $H_2 := \varkappa_4$ we obtain:
\begin{equation} \label{eq:inverse:H}
  \fl \lambda_n (\varphi'(\pi, \lambda_n) + H \varphi(\pi, \lambda_n)) = H_1 \varphi'(\pi, \lambda_n) + H_2 \varphi(\pi, \lambda_n), \quad n \in \mathbb{Z}_{+}
\end{equation}
for some constants $H$, $H_1$ and $H_2$.
From~(\ref{eq:A_B}) we have:
\begin{equation*}
  H H_1 - H_2 = \rho.
\end{equation*}
Hence we have proved
\begin{thm} \label{thm:main_theorem}
  For the sequences $\{\lambda_n\}$ and $\{\gamma_n\}$ $(n \in \mathbb{Z}_{+})$ to be the spectral data of a problem of the type~(\ref{eq:main})--(\ref{eq:boundary2}) it is necessary and sufficient to satisfy conditions~(\ref{eq:spectral_data1})--(\ref{eq:spectral_data2}).
\end{thm}
\begin{exmp}
  Assume that the spectral data of some eigenvalue problem of the form~(\ref{eq:main})--(\ref{eq:boundary2}) is the following:
\begin{equation*}
  \lambda_0 = 0, \quad \lambda_1 = \frac{1}{4}, \qquad \lambda_n = (n-1)^2, \quad n \ge 2,
\end{equation*}
\begin{equation*}
  \gamma_0 = \pi, \qquad \gamma_n = \frac{\pi}{2}, \quad n \ge 1.
\end{equation*}
Then from~(\ref{eq:def:F}) we have:
\begin{equation*}
  F(x,t) = \frac{2}{\pi} \cos \frac{x}{2} \cos \frac{t}{2}.
\end{equation*}
Solving the equation~(\ref{eq:main_equation}) and then using the relation~(\ref{eq:K_x_x}) we obtain:
\begin{equation*}
  \fl K(x,t) = - \frac{2 \cos \frac{x}{2} \cos \frac{t}{2}}{\pi + x + \sin x}, \quad q(x) = \frac{2(\pi + x) \sin x + 4(1 + \cos x)}{(\pi + x + \sin x)^2}, \quad h = -\frac{2}{\pi}.
\end{equation*}
In order to reconstruct the second boundary condition, we construct the solution $\varphi(x, \lambda)$ using~(\ref{eq:inverse:phi}):
\begin{equation*}
  \varphi(x, \lambda) = \left\{
\begin{array}{ll}
  \cos sx - \frac{\displaystyle 4s \sin sx (1 + \cos x) - \cos sx \sin x}{\displaystyle (4s^2 - 1)(\pi + x + \sin x)}, & \lambda \ne \frac{1}{4}, \\
  \frac{\displaystyle \pi \cos \frac{x}{2}}{\displaystyle \pi + x + \sin x}, & \lambda = \frac{1}{4}.
\end{array}
\right.
\end{equation*}
Then from~(\ref{eq:inverse:H}) we have:
\begin{equation*}
  8 \pi H n^4 - (2 \pi H + 8 \pi H_2 + 1) n^2 + H_1 + 2 \pi H_2 = 0, \quad n = 0, 1, \dots.
\end{equation*}
From these equalities we finally calculate the coefficients of the second boundary condition:
\begin{equation*}
  H = 0, \quad H_1 = \frac{1}{4}, \quad H_2 = - \frac{1}{8 \pi}.
\end{equation*}
\end{exmp}

\section{On two problems with common parameter dependent boundary condition}
\label{sec:two_spectra}

Consider two eigenvalue problems for the equation
\begin{equation}
\label{eq:reconstruction2:main}
  -y''(x) + q(x)y(x) = \lambda y(x)
\end{equation}
with boundary conditions
\begin{eqnarray}
  y'(0) - hy(0) = 0, & \qquad \lambda (y'(\pi) + H y(\pi)) = H_1 y'(\pi) + H_2 y(\pi), \label{eq:reconstruction2:boundary1} \\
  y'(0) - \widetilde{h}y(0) = 0, & \qquad \lambda (y'(\pi) + H y(\pi)) = H_1 y'(\pi) + H_2 y(\pi), \label{eq:reconstruction2:boundary2}
\end{eqnarray}
where $q(x) \in \mathscr{L}_2(0,\pi)$ is a real-valued function, $h, \widetilde{h}, H, H_1, H_2 \in \mathbb{R}$ and
\begin{equation*}
  \rho := H H_1 - H_2 > 0.
\end{equation*}
We can assume without loss of generality that $h < \widetilde{h}$. Denote by $\lambda_0 < \lambda_1 < \lambda_2 < \dots$ and $\mu_0 < \mu_1 < \mu_2 < \dots$ the eigenvalues of the problems~(\ref{eq:reconstruction2:main}),~(\ref{eq:reconstruction2:boundary1}) and~(\ref{eq:reconstruction2:main}),~(\ref{eq:reconstruction2:boundary2}), respectively.

Let $\varphi(x, \lambda)$ and $\psi(x, \lambda)$ be the solutions of equation~(\ref{eq:reconstruction2:main}) satisfying
\begin{equation*}
  \varphi(0, \lambda) = 1, \quad \varphi'(0, \lambda) = h, \qquad \psi(0, \lambda) = 1, \quad \psi'(0, \lambda) = \widetilde{h}.
\end{equation*}
Eigenvalues of the problems~(\ref{eq:reconstruction2:main}),~(\ref{eq:reconstruction2:boundary1}) and~(\ref{eq:reconstruction2:main}),~(\ref{eq:reconstruction2:boundary2}) coincide with the zeros of the functions
\begin{equation*}
  \Phi(\lambda) := \lambda (\varphi'(\pi, \lambda) + H \varphi(\pi, \lambda)) - H_1 \varphi'(\pi, \lambda) - H_2 \varphi(\pi, \lambda),
\end{equation*}
\begin{equation*}
  \Psi(\lambda) := \lambda (\psi'(\pi, \lambda) + H \psi(\pi, \lambda)) - H_1 \psi'(\pi, \lambda) - H_2 \psi(\pi, \lambda),
\end{equation*}
respectively. We denote $f(x, \lambda) = \psi(x, \lambda) + m(\lambda) \varphi(x, \lambda)$ and choose $m(\lambda)$ such that
\begin{equation*}
  \lambda (f'(\pi, \lambda) + H f(\pi, \lambda)) - H_1 f'(\pi, \lambda) - H_2 f(\pi, \lambda) = 0.
\end{equation*}
Then
\begin{equation*}
  m(\lambda) = - \frac{\Psi(\lambda)}{\Phi(\lambda)}.
\end{equation*}
Using Green's formula, we can write
\begin{equation*}
  \int_0^{\pi} f(x, \lambda) f(x, \mu) \rmd x = - \frac{\rho f(\pi, \lambda) f(\pi, \mu)}{(H_1 - \lambda)(H_1 - \mu)} + (h - \widetilde{h}) \frac{m(\lambda) - m(\mu)}{\lambda - \mu}.
\end{equation*}
From here when $\mu \to \lambda$ it follows that:
\begin{equation*}
  \int_0^{\pi} f^2(x, \lambda) \rmd x + \frac{\rho f^2(\pi, \lambda)}{(H_1 - \lambda)^2} = (h - \widetilde{h}) \dot{m}(\lambda).
\end{equation*}
Since the left-hand side of the last equality is always positive, the function $m(\lambda)$ monotonically decreases in the set $\mathbb{R} \setminus \{ \lambda_n | n \in \mathbb{Z}_{+} \}$. Therefore, zeros and poles of $m(\lambda)$ interlace and according to~(\ref{eq:asymptotics:s_n}) we obtain:
\begin{equation*}
  \lambda_0 < \mu_0 < \lambda_1 < \mu_1 < \lambda_2 < \mu_2 < \dots.
\end{equation*}
Now we use Green's formula again:
\begin{eqnarray*}
  \fl (\lambda - \lambda_n) \int_0^{\pi} f(x, \lambda) \varphi(x, \lambda_n) \rmd x = \left. (f(x, \lambda) \varphi'(x, \lambda_n) - f'(x, \lambda) \varphi(x, \lambda_n)) \right|_0^{\pi} \\
  \lo{=} - \frac{\rho (\lambda - \lambda_n) f(\pi, \lambda) \varphi(\pi, \lambda_n)}{(H_1 - \lambda_n)(H_1 - \lambda)} - (h - \widetilde{h}) = - \frac{\rho (\lambda - \lambda_n) \psi(\pi, \lambda) \varphi(\pi, \lambda_n)}{(H_1 - \lambda_n)(H_1 - \lambda)} \\
  + \frac{\rho (\lambda - \lambda_n) \varphi(\pi, \lambda) \varphi(\pi, \lambda_n)}{(H_1 - \lambda_n)(H_1 - \lambda)} \frac{\Psi(\lambda)}{\Phi(\lambda)} - (h - \widetilde{h}).
\end{eqnarray*}
On the other hand
\begin{eqnarray*}
  \fl (\lambda - \lambda_n) \int_0^{\pi} f(x, \lambda) \varphi(x, \lambda_n) \rmd x = (\lambda - \lambda_n) \int_0^{\pi} \psi(x, \lambda) \varphi(x, \lambda_n) \rmd x \\
  - \frac{(\lambda - \lambda_n) \Psi(\lambda)}{\Phi(\lambda)} \int_0^{\pi} \varphi(x, \lambda) \varphi(x, \lambda_n) \rmd x.
\end{eqnarray*}
When $\lambda_n \to \lambda$ from these two equalities we obtain:
\begin{equation*}
  \gamma_n = (h - \widetilde{h}) \frac{\dot{\Phi}(\lambda_n)}{\Psi(\lambda_n)}.
\end{equation*}
Thus we expressed norming constants by two spectra. We shall use this expression to solve the problem of reconstruction of differential operator by two of its spectra.
Using Theorem~\ref{thm:asymptotics} and Lemma~\ref{lem:Marchenko} from the last equality we obtain the following asymptotic estimation:
\begin{equation*}
  \sqrt{\lambda_n} - \sqrt{\mu_n} := \frac{h - \widetilde{h}}{(n - 1) \pi} + \frac{\zeta_n}{n^2}, \quad \{\zeta_n\} \in l_2.
\end{equation*}

\section{Reconstruction by two spectra}
\label{sec:reconstruction2}

Let two sequences of real numbers $\{\lambda_n\}$ and $\{\mu_n\}$ ($n \in \mathbb{Z}_{+}$) with the following properties be given:
\begin{equation}
\label{eq:reconstruction2:asymptotics}
  \sqrt{\lambda_n} = n - 1 + \frac{\omega}{n \pi} + \frac{\zeta_n}{n}, \quad \{\zeta_n\} \in l_2,
\end{equation}
\begin{equation}
\label{eq:reconstruction2:asymptotics2}
  \sqrt{\mu_n} - \sqrt{\lambda_n} = \frac{\sigma}{n \pi} + \frac{\zeta'_n}{n^2}, \quad \sigma > 0, \quad \{\zeta'_n\} \in l_2,
\end{equation}
\begin{equation*}
  \lambda_0 < \mu_0 < \lambda_1 < \mu_1 < \lambda_2 < \mu_2 < \dots.
\end{equation*}
We define functions
\begin{equation*}
  \Phi(\lambda) := -\pi(\lambda-\lambda_0)(\lambda-\lambda_1) \prod_{n=2}^{\infty} \frac{\lambda_n - \lambda}{(n-1)^2},
\end{equation*}
\begin{equation*}
  \Psi(\lambda) := -\pi(\lambda-\mu_0)(\lambda-\mu_1) \prod_{n=2}^{\infty} \frac{\mu_n - \lambda}{(n-1)^2},
\end{equation*}
\begin{equation*}
  m(\lambda) := -\frac{\Psi(\lambda)}{\Phi(\lambda)}
\end{equation*}
and put:
\begin{equation*}
  \gamma_n := - \sigma \frac{\dot{\Phi}(\lambda_n)}{\Psi(\lambda_n)}.
\end{equation*}

Since zeros of $\Phi(\lambda)$ are simple, we get:
\begin{equation*}
  \Res_{\lambda=\lambda_n} m(\lambda) = -\frac{\Psi(\lambda_n)}{\dot{\Phi}(\lambda_n)} = \frac{\sigma}{\gamma_n}.
\end{equation*}
Lemma~\ref{lem:Marchenko} implies: $\lim_{\lambda \to -\infty} m(\lambda) = -1$. Hence
\begin{equation*}
  m(\lambda) = -1 + \sum_{n=0}^{\infty} \frac{\sigma}{\gamma_n (\lambda-\lambda_n)}.
\end{equation*}
Using Lemma~\ref{lem:Marchenko} again we calculate:
\begin{equation*}
  \dot{\Phi}(\lambda_n) = (-1)^n (n-1)^2 \left( \frac{\pi}{2} + \frac{\xi_n}{n} \right), \quad \{\xi_n\} \in l_2,
\end{equation*}
\begin{equation*}
  \Psi(\lambda_n) = (-1)^{n+1} (n-1)^2 \left( \sigma + \frac{\xi'_n}{n} \right), \quad \{\xi'_n\} \in l_2,
\end{equation*}
\begin{equation*}
  \gamma_n = \frac{\pi}{2} + \frac{\zeta_n}{n}, \quad \{\zeta_n\} \in l_2.
\end{equation*}

According to Theorem~\ref{thm:main_theorem} we can uniquely construct such a real-valued function $q(x) \in \mathscr{L}_2(0,\pi)$ and numbers $h, H, H_1, H_2 \in \mathbb{R}$ that the numbers $\{\lambda_n\}$ and $\{\gamma_n\}$ are the eigenvalues and the norming constants of problem $\mathcal{P}(q,h,H,H_1,H_2)$. We put:
\begin{equation*}
  \widetilde{h} := h + \sigma
\end{equation*}
and denote by $\tau_n$ the eigenvalues of problem $\mathcal{P}(q,\widetilde{h},H,H_1,H_2)$. Let $\varphi(x, \lambda)$ and $\psi(x, \lambda)$ be the solutions of equation~(\ref{eq:reconstruction2:main}) satisfying the conditions
\begin{equation*}
  \varphi(0, \lambda) = 1, \quad \varphi'(0, \lambda) = h, \qquad \psi(0, \lambda) = 1, \quad \psi'(0, \lambda) = \widetilde{h}.
\end{equation*}
Then zeros and poles of the function
\begin{equation*}
  \widetilde{m}(\lambda) := -\frac{\lambda (\psi'(\pi, \lambda) + H \psi(\pi, \lambda)) - H_1 \psi'(\pi, \lambda) - H_2 \psi(\pi, \lambda)}{\lambda (\varphi'(\pi, \lambda) + H \varphi(\pi, \lambda)) - H_1 \varphi'(\pi, \lambda) - H_2 \varphi(\pi, \lambda)}
\end{equation*}
coincide with $\{\lambda_n\}$ and $\{\tau_n\}$, respectively. The numbers $\{\lambda_n\}$ and $\{\tau_n\}$ interlace, $\lim_{\lambda \to -\infty} \widetilde{m}(\lambda) = -1$ and
\begin{equation*}
  \Res_{\lambda=\lambda_n} \widetilde{m}(\lambda) = \frac{\widetilde{h} - h}{\gamma_n} = \frac{\sigma}{\gamma_n}.
\end{equation*}
Using residue calculus we obtain:
\begin{equation*}
  \widetilde{m}(\lambda) = -1 + \sum_{n=0}^{\infty} \frac{\sigma}{\gamma_n (\lambda-\lambda_n)},
\end{equation*}
i.e. $\widetilde{m}(\lambda) \equiv m(\lambda)$ and therefore zeros of $\widetilde{m}(\lambda)$ and $m(\lambda)$ coincide: $\tau_n = \mu_n$. Thus we have proved
\begin{thm} \label{thm:main_theorem2}
  For the sequences $\{\lambda_n\}$ and $\{\mu_n\}$ $(n \in \mathbb{Z}_{+})$ to be the spectra of two problems of the type~(\ref{eq:main})--(\ref{eq:boundary2}) with common parameter dependent boundary condition it is necessary and sufficient to satisfy conditions~(\ref{eq:reconstruction2:asymptotics})--(\ref{eq:reconstruction2:asymptotics2}) and interlace.
\end{thm}

\ack
  The author is grateful to Professor H.~M.~Huseynov for the problem statement and for useful discussions.

\end{document}